\documentclass[11pt]{article}
\usepackage{amssymb,float,amsthm,hyperref,xcolor,amsmath,tipa,graphicx}
\hypersetup{
    colorlinks,
    citecolor=violet,
    filecolor=blue,
    linkcolor=blue,
    urlcolor=violet
}
\usepackage{array}
\usepackage{tikz}
\newtheorem{theorem}{Theorem}[section]
\newtheorem{lemma}[theorem]{Lemma}

\newtheorem{proposition}[theorem]{Proposition}

\theoremstyle{definition}
\newtheorem{remark}[theorem]{Remark}
\newtheorem{example}[theorem]{Example}
\newtheorem{question}[theorem]{Open Question}
\newtheorem{definition}[theorem]{Definition}

\newcommand{\half}{{\textstyle\frac12}}

\numberwithin{equation}{section}
\usepackage[margin=1in]{geometry}
\counterwithin{figure}{section}
\usepackage{setspace}

\title{Pattern avoidance in permutations and their rotations}
\author{ \"{O}mer E\u{g}ecio\u{g}lu\thanks{Department of Computer Science, University of California Santa Barbara, Santa Barbara, CA 93106, United States of America. Email: {\tt omer@cs.ucsb.edu}}
\and 
Collier Gaiser\thanks{Department of Mathematics, Community College of Aurora, Aurora, CO 80011,  United States of America. Email: {\tt colliergaiser@gmail.com}}
\and  Mei Yin\thanks{Department of Mathematics, University of Denver, Denver, CO 80208, United States of America. Email: {\tt mei.yin@du.edu} }} 
\date{}

\begin{document}
\maketitle

\begin{abstract}
A rotation of a permutation is a new permutation obtained by moving the first several terms of the permutation to the end of the permutation. A circular permutation is the set of all rotations of a permutation. The enumerations of permutations and circular permutations avoiding patterns of length three and four are well studied. 
In this paper, we provide exact formulas for the number of permutations whose first $k\geq 2$ 
rotations all avoid a given pattern of length three, as well as the number of permutations whose first three rotations respectively avoid the rotations of a given pattern of length three. In contrast to permutations and circular permutations avoiding patterns of length three, the Wilf-equivalence classes under study are entirely determined by complements and reverses. We also classify and enumerate permutations whose first two rotations avoid different patterns of length three.
\end{abstract}


\section{Introduction}\label{Section:Intro}

For $n\in\mathbb{N}:=\{1,2,\ldots\}$, a (linear) permutation $p=(p(1),p(2),\ldots, p(n))$ 
of $\{1,2,\ldots,n\}$ is a sequence of $n$ distinct integers in $\{1,2,\ldots,n\}$. We use $S_n$ to denote the set of permutations of $\{1,2,\ldots,n\}$. 
When there is no confusion, we simply use the one-line notation and write  $p=p(1)p(2)\cdots p(n)$.

For any $p\in S_n$ and $q\in S_m$, if there exist $1\leq i_1<i_2<\cdots<i_m\leq n$ such that, for all $1\leq a<b\leq m$, we have $p(i_a)<p(i_b)$ if and only if $q(a)<q(b)$, 
then we say that $p$ contains $q$ as a pattern and 
that $p(i_1)p(i_2)\cdots p(i_m)$ is a $q$ pattern.
A permutation
$p$ is said to avoid $q$ if $p$ does not contain $q$ as a pattern. For example, the permutation $p=12453\in S_5$ contains the pattern $132$ because $p(1)p(3)p(5)=143$ is a $132$ pattern; however, $p$ avoids the pattern $321$.
We define $S_n(q)$ to be the set of permutations in $S_n$ avoiding $q$.

The interest in the study of 
pattern avoidance in permutations can be traced back to 
stack-sortable permutations in computer science \cite[Section 2.1]{Kitaev2011}. 
One of the earliest results is the enumeration of permutations avoiding $\sigma\in S_3$, 
i.e., patterns of length three. D. Knuth proved that the number 
of permutations in $S_n$ avoiding any given pattern of length three is 
counted by the Catalan number $C_n$
(see also \cite[Theorem 4.7]{Bona2022}).

\begin{theorem} {\rm \cite[p. 238]{Knuth1973}} \label{Theorem:ClassicalSingle3}
For all $n\geq1$ and $\sigma\in S_3$, we have
\[
|S_n(\sigma)|=C_n=\frac{1}{n+1} \binom{2n}{n}.
\]
\end{theorem}

Given $m\in\mathbb{N}$ and $q_1,q_2\in S_m$, we say that $S_n(q_1)$ and $S_n(q_2)$ 
are \textit{Wilf-equivalent} if $|S_n(q_1)|=|S_n(q_2)|$ for all $n\in\mathbb{N}$ . In the study of permutation patterns, it is 
standard to call $q_1$ and $q_2$ Wilf-equivalent if $|S_n(q_1)|=|S_n(q_2)|$ for all $n\in\mathbb{N}$ (see \cite[Chapter 1]{Kitaev2011}). However, since we are going to look at pattern avoidance under a variety of restrictions, we include the target set of permutations when we discuss Wilf-equivalence classes. By Theorem~\ref{Theorem:ClassicalSingle3}, there is exactly one Wilf-equivalence class for $S_n(q)$ with $q\in S_3$.

For any permutation $p\in S_n$, let $p^{(1)}=p$ be the \textit{first (positive) rotation} of $p$ and, for all $k\in\{2,3,\ldots,n\}$, 
let \[p^{(k)}=p(k)p(k+1)\ldots p(n)p(1)p(2)\ldots p(k-1)\] be the $k$th \textit{(positive) 
rotation} of $p$. Define  
$[p]:=\{p^{(1)},p^{(2)},\ldots,p^{(n)}\}
$
as the corresponding \textit{circular permutation}. Hence, a circular permutation is the set of all rotations of a given permutation. In~\cite{GLW2018}, permutation rotations are called circular shifts. Circular permutations have also been given other names, such as cyclic permutations \cite{Li2022} and cyclic arrangements \cite{Vella2003}; however, since the name cyclic permutations is also used for permutations which have only one cycle in the cycle notation~\cite{AE2014, BC2019,Huang2019}, we will call $[p]$ with $p\in S_n$ circular permutations in this paper. Note that, we use $p$ and $p^{(1)}$ interchangeably in this paper.

\begin{example}
    Let $p=2756431\in S_7$. Then we have $p^{(2)}=7564312$, $p^{(3)}=5643127$, $p^{(4)}=6431275$, $p^{(5)}=4312756$, $p^{(6)}=3127564$, $p^{(7)}=1275643$, and \[[2756431]=\{2756431, 7564312, 5643127, 6431275, 4312756, 3127564, 1275643\}.\]
\end{example}

We say that a circular permutation $[p]$ contains $q$ as a pattern if there 
exists $p'\in[p]$ such that $p'$ contains $q$ as a pattern. Let $\text{Av}_n[q]$ 
be the set of circular permutations of length $n$ that do not contain $q$ as a pattern. As noted by several papers \cite{Callan2002,DELMSSS2022}, due to the definition of circular permutations, we have $
\text{Av}_n[123]=\{[(n,n-1,\ldots,1)]\}$ and $
\text{Av}_n[321]=\{[(1,2,\ldots,n)]\}$. Interestingly, Callan \cite{Callan2002} and Vella \cite{Vella2003} proved the 
following formulas:
\[
|\text{Av}_n[1234]|=|\text{Av}_n[1432]|=2^n+1-2n-\binom{n}{3},
\]
\[
|\text{Av}_n[1243]|=|\text{Av}_n[1342]|=2^{n-1}-n+1,
\]
and
\[
|\text{Av}_n[1324]|=|\text{Av}_n[1423]|=F_{2n-3},
\]
where $F_i$ is the $i$th Fibonacci number with $F_1=F_2=1$ and $F_{i}=F_{i-1}+F_{i-2}$ for all $i\geq3$.

Other recent works on pattern avoidance study circular permutations avoiding pairs of patterns of length four \cite{DELMSSS2022}, avoiding vincular patterns \cite{ES2021,Li2022,MS2021}, avoiding a pattern of length four and a pattern of length $k$ \cite{MS2022}, as well as permutations whose cycle notation has a single cycle \cite{ABBGJ2023,ABBGJ2025}. There are also results on pattern containment in circular permutations \cite{DELMSSS2022,GLW2018} and stack-sorting maps which avoid all rotations of a given pattern \cite{ZB2026}. The goal of the current paper is to investigate this rotation process more in-depth and enumerate pattern-avoiding permutations when only $k$, $2\leq k\leq n$, rotations are considered. 

Given $m,n,k\in\mathbb{N}$ and $q_1,q_2,\ldots,q_k\in S_m$, let $S_{n}^{(k)}(q_1:q_2:\cdots:q_k)$ be the set of permutations $p\in S_n$ such that $p^{(i)}$ avoids $q_i$ for all $i\in\{1,2,\ldots,k\}$. When $q_1=q_2=\cdots=q_k=q$ for some $q$, we simply write $S_n^{(k)}(q)=S_n^{(k)}(q_1:q_2:\cdots:q_k)$. Notice that, for any pattern $q$, we have $|S_n^{(n)}(q)|=n|\text{Av}_n[q]|$ and 
\[
S_n^{(n)}(q)\subseteq S_n^{(n-1)}(q)\subseteq\cdots\subseteq S_n^{(2)}(q)\subseteq S_n^{(1)}(q)=S_n(q).
\]

\begin{example}
    Let $p=654231\in S_6$. Then $p^{(1)}=p=654231$ and $p^{(2)}=542316$ both avoid the pattern $132$. Hence $p\in S_6^{(2)}(132)$. 
However, in $p^{(3)}=423165$, $465$ is a $132$ pattern, therefore $p\notin S_6^{(3)}(132)$.
\end{example}

By Theorem~\ref{Theorem:ClassicalSingle3}, there is one Wilf-equivalence class for $S_n^{(1)}(q)$ with $q\in S_3$. 
Since $\text{Av}_n[123]=\{[(n,n-1,\ldots,1)]\}$ and $
\text{Av}_n[321]=\{[(1,2,\ldots,n)]\}$, there is also one Wilf-equivalence class for $S_n^{(n)}(q)$ with $q\in S_3$. In this paper, however, we show that there are two Wilf-equivalence classes for $S_n^{(k)}(q)$ with $q\in S_3$ when $k\geq2$. 
Enumerating permutations by fixing leading terms as in our earlier work \cite{EGY2026}, 
we find the exact formulas for $|S_n^{(k)}(q)|$ for $q\in S_3$ and $2\leq k\leq n$. In particular,
\[
|S_n^{(k)}(123)|=|S_n^{(k)}(321)|=2^{n-k+1}+k-2,
\]
and
\[
|S_n^{(k)}(213)|=|S_n^{(k)}(231)|=|S_n^{(k)}(312)|=|S_n^{(k)}(132)|=k-2+\sum_{i=0}^{n-k+1}C_i,
\]
where $C_i$ is the $i$th Catalan number. The two Wilf-equivalence classes are entirely determined by complements 
and reverses. These interpretations will be covered in Section~\ref{Section:Pre}. As is usual with pattern avoidance related work, in most instances our proofs involve case-by-case analysis.

Wilf-equivalence classes may also be determined by complements and reverses for permutations avoiding {\it rotational 3-chains}. Given $n\in\mathbb{N}$ and $q\in S_3$, we use $S_n^{(3)}(q:q^{(2)}:q^{(3)})$ to denote the set of permutations $p\in S_n$ such that $p$ avoids $q$, $p^{(2)}$ avoids $q^{(2)}$, and $p^{(3)}$ avoids $q^{(3)}$. Since $q^{(2)}$ and $q^{(3)}$ are rotations of $q$, we call $(q:q^{(2)}:q^{(3)})$ a rotational 3-chain. We show that, for all $n\geq3$, we have
\[
\begin{split}
    |S_n^{(3)}(123:231:312)|=|S_n^{(3)}(321:213:132)|=&|S_n^{(3)}(231:312:123)|\\=&|S_n^{(3)}(213:132:321)|=5n-10,
\end{split}
\]
and
\[
    |S_n^{(3)}(312:123:231)|=|S_n^{(3)}(132:321:213)|=\half (n+7)(n-2).
\]

We further study permutations $p$ such that $p^{(1)}$ and $p^{(2)}$ avoid two different patterns $q_1$ and $q_2$, of length three, respectively. We call $(q_1:q_2)$ a \textit{2-chain}. 
We show that for many 2-chains $(q_1:q_2)$ of length three and $n\geq2$, we have $|S_n^{(2)}(q_1:q_2)|=|S_n(q_1,q_2)|$, 
even though $S_n^{(2)}(q_1:q_2)\neq S_n(q_1,q_2)$, where $S_n(q_1,q_2)$ is the set of permutations $p\in S_n$ such that $p$ avoids both $q_1$ and $q_2$ simultaneously. If $|S_n^{(2)}(q_1:q_2)|\neq |S_n(q_1,q_2)|$ with $q_1\neq q_2$ and not both of them are monotone, then there are three possible expressions for $|S_n^{(2)}(q_1:q_2)|$ with $n\geq2$. 
These are 
\[
2n-2, ~~(n+2)\cdot2^{n-3}, ~~\mbox{and} ~~ n + 2 \binom{n}{3}. 
\]
The last two expressions are especially interesting because they enumerate two known 
integer sequences A045623 and A116731 respectively, in OEIS \cite{OEIS}, and hence offer new interpretations 
for these sequences in terms of permutation rotations.

We would like to point out that our interest in studying permutations such that the first few rotations avoid certain patterns was partially inspired by several recent 
studies on pattern avoidance making use of the group structure of the symmetric group. B\'{o}na and Smith \cite{BonaSmith2019} studied permutations $p$ such that both $p$ and $p^2$ avoid a pattern of length three. Burcroff and Defant \cite{BurcroffDefant2020} went one step further and studied permutations whose powers all avoid a given pattern. Archer and Geary \cite{ArcherGeary2023} studied a more generalized scenario in which permutation powers avoid a chain of patterns.

This paper is organized as follows. In Section~\ref{Section:Pre}, we introduce negative rotations of a permutation and prove a generalized result on complements and reverses of permutations. In Section~\ref{Section:SinglePattern3}, we enumerate permutations whose first $k$ rotations all avoid a given pattern of length three. In Section~\ref{Section:Rot3Chain}, we enumerate permutations whose first three rotations avoid a rotational 3-chain. Permutations avoiding $2$-chains are studied in Section~\ref{Section:2Chain}. In Section~\ref{Section:Conclude} we show that, for certain cases, the Wilf-equivalence classes for permutation rotations when the patterns are of length four are also entirely determined by complements and reverses.

\section{Preliminaries}\label{Section:Pre}
We first define negative rotations and show how they relate to positive rotations.

\begin{definition}
    For all $p\in S_n$, let $p^{(1)'}=p$ be the \textit{first negative rotation} of $p$, and for $k\in\{2,3,\ldots,n\}$, let
\[p^{(k)'}=p(n-k+2)p(n-k+3)\ldots p(n) p(1) p(2)\ldots p(n-k+1)\] be the \textit{$k$th negative rotation} of $p$.
\end{definition}

Given $k,m,n\in\mathbb{N}$ with $k\leq n$ and $q_1,q_2,\ldots,q_k\in S_m$, let $S_n^{(k)'}(q_1:q_2:\cdots;q_k)$ be the set of permutations $p\in S_n$ such that $p^{(i)'}$ avoids $q_i$ for all $i\in\{1,2,\ldots,k\}$.
\begin{lemma}\label{Lemma:ClockCounterclock}
Let $k,m,n\in\mathbb{N}$ with $k\leq n$ and $q_1,q_2,\ldots,q_k\in S_m$. Then 
\[
|S_n^{(k)}(q_1:q_2:\cdots:q_k)|=|S_n^{(k)'}(q_k:q_{k-1}: \cdots:q_1)|.
\]
\end{lemma}
\begin{proof}
Let $f:S_n^{(k)}(q_1:q_2:\cdots:q_k)\to S_n^{(k)'}(q_k:q_{k-1}:\cdots:q_1)$ be defined by $f(p)=p^{(k)}$. We need to show that $p^{(k)}\in S_n^{(k)'}(q_k:q_{k-1}:\cdots:q_1)$ for all $p\in S_n^{(k)}(q_1:q_2:\cdots:q_k)$. Let $p\in S_n^{(k)}(q_1:q_2:\cdots:q_k)$. Then $p^{(i)}$ avoids $q_i$ for all $i\in\{1,2,\ldots,k\}$. By the definition of negative rotations, we have $(p^{(k)})^{(j)'}=p^{(k-j+1)}$ for all $j\in\{1,2,\ldots,k\}$. It follows that $(p^{(k)})^{(j)'}$ avoids $q_{k-j+1}$ for all $j\in\{1,2,\ldots,k\}$. Hence, $p^{(k)}\in S_n^{(k)'}(q_k:q_{k-1}:\cdots:q_1)$.

It remains to show that $f$ is a bijection. By the definition of permutation rotations, if $p_1^{(k)}\neq p_2^{(k)}$, then $p_1\neq p_2$. Hence, $f$ is one-to-one. Next take $\sigma\in S_n^{(k)'}(q_k:q_{k-1}:\cdots:q_1)$. 
Using a similar argument as in the previous paragraph, we have $\sigma^{(k)'}\in S_n^{(k)}(q_1:q_2:\cdots:q_k)$. Since $f(\sigma^{(k)'})=\sigma$, $f$ is also onto.
\end{proof}
Now we define the complement and reverse of a permutation. We will then use these to reduce the number of cases we need for enumeration in later sections.

\begin{definition}\label{Definition:Complements}
Let $n\in\mathbb{N}$. For any $p\in S_n$, the {\em complement} $p^c$ of $p$ is a permutation in $S_n$ defined by setting 
$p^c(i)=n+1-p(i)$ for all $i\in\{1,2,\ldots,n\}$. The {\em reverse} $p^r$ of $p$ is a permutation in $S_n$ defined by setting $p^r(i)=p(n+1-i)$ for all $i\in\{1,2,\ldots,n\}$. 
\end{definition}

It is a standard observation that $p$ avoids $q$ if and only if $p^c$ avoids $q^c$; and $p$ avoids $q$ if and only if $p^r$ avoids $q^r$.

\begin{lemma}\label{Lemma:DirectionInverse}
Let $n\in\mathbb{N}$. For all $p\in S_n$ and $k\leq n$, we have $(p^{(k)})^r=(p^r)^{(k)'}$.
\end{lemma}
\begin{proof}
    Let $p=p(1)p(2)\cdots p(n)\in S_n$. Then
    \[
    p^{(k)}=p(k)p(k+1)\cdots p(n)p(1)p(2)\ldots p(k-1),
    \]
    and hence
    \[
    (p^{(k)})^r=p(k-1)p(k-2)\cdots p(1)p(n)p(n-1)\ldots p(k).
    \]
At the same time, we have $p^r=p(n)p(n-1)\cdots p(1)$ and hence
\[
(p^r)^{(k)'}=p(k-1)p(k-2)\ldots p(1)p(n)p(n-1)\ldots p(k).
\]
Hence $(p^{(k)})^r=(p^r)^{(k)'}$.
\end{proof}

\begin{theorem}\label{Theorem:ComplementReverse}
Let $k,m,n\in\mathbb{N}$ with $k\leq n$ and $q_1,q_2,\ldots,q_k\in S_m$. Then
    \[
    |S_n^{(k)}(q_1:q_2:\cdots:q_k)|=|S_n^{(k)}(q_1^c:q_2^c:\cdots:q_k^c)|=|S_n^{(k)}(q_k^r:q_{k-1}^r:\cdots:q_1^r)|.
    \]
\end{theorem}

\begin{proof}
We first prove that $|S_n^{(k)}(q_1:q_2:\cdots:q_k)|=|S_n^{(k)}(q_1^c:q_2^c:\cdots:q_k^c)|$. Since rotations only change the location of each term of a permutation, we have $\left(p^{(k)}\right)^c=\left(p^c\right)^{(k)}$ for all $p\in S_n$. It follows that $p\in S_n^{(k)}(q_1:q_2:\cdots:q_k)$ if and only if $p^c\in S_n^{(k)}(q_1^c:q_2^c:\cdots:q_k^c)$. Hence $|S_n^{(k)}(q_1:q_2:\cdots:q_k)|=|S_n^{(k)}(q_1^c:q_2^c:\cdots:q_k^c)|$.

Now we prove that $|S_n^{(k)}(q_1:q_2:\cdots:q_k)|=|S_n^{(k)}(q_k^r:q_{k-1}^r:\cdots:q_1^r)|$. By Lemma~\ref{Lemma:ClockCounterclock}, it suffices to prove that $|S_n^{(k)'}(q_k:q_{k-1}:\cdots:q_1)|=|S_n^{(k)}(q_k^r:q_{k-1}^r:\cdots:q_1^r)|$. Let \[g:S_n^{(k)}(q_k^r:q_{k-1}^r:\cdots:q_1^r)\to S_n^{(k)'}(q_k:q_{k-1}:\cdots:q_1)\] 
be defined by 
$g(p)=p^r$. Suppose $p\in S_n^{(k)}(q_k^r:q_{k-1}^r:\cdots:q_1^r)$. Then $p^{(i)}$ avoids $q_{k-i+1}^r$ for all $i\in\{1,2,\ldots,k\}$. If follows that $(p^{(i)})^r$ avoids $q_{k-i+1}$ for all $i\in\{1,2,\ldots,k\}$. By Lemma~\ref{Lemma:DirectionInverse}, we have $(p^r)^{(i)'}=(p^{(i)})^r$ for all $i\in\{1,2,\ldots,k\}$. This implies that $(p^r)^{(i)'}$ avoids $q_{k-i+1}$ for all $i\in\{1,2,\ldots,k\}$, nd 
hence  $p^r\in S_n^{(k)'}(q_k:q_{k-1}:\cdots:q_1)$. $g$ is one-to-one since  $p_1^r=p_2^r$ if and only if $p_1=p_2$. We also have that $g$ is onto because $g(p^r)=p$ for all $p\in S_n^{(k)'}(q_k:q_{k-1}:\cdots:q_1)$. 
Therefore $g$ is a bijection and $|S_n^{(k)'}(q_k:q_{k-1}:\cdots:q_1)|=|S_n^{(k)}(q_k^r:q_{k-1}^r:\cdots:q_1^r)|$.
\end{proof}

When $q_1=q_2=\cdots=q_k=q$ for some $q\in S_m$, Theorem~\ref{Theorem:ComplementReverse} says that $|S_n^{(k)}(q)|=|S_n^{(k)}(q^c)|=|S_n^{(k)}(q^r)|$. This observation has been widely used in pattern avoidance studies to classify Wilf-equivalence classes (see, for example, \cite[p.~151]{Bona2022}). Theorem~\ref{Theorem:ComplementReverse} 
makes this argument stronger in terms of reverses because when $q_1,q_2,\ldots,q_k$ are different, we need reverses of both individual patterns and the order of the patterns.

Finally, we define subpermutations of a permutation that will be used in our proofs.

\begin{definition}
For any finite set $X\subseteq\mathbb{N}$, let $S_X$ be the set of permutations of $X$. When $X=\{1,2,\ldots,n\}$, we write $S_n=S_{\{1,2,\ldots,n\}}$. Let $A$ and $B$ be two finite subsets of $\mathbb{N}$ with $A\subseteq B$, $\sigma\in S_A$, and $\tau\in S_B$. We say that $\sigma$ is a 
{\em subpermutation} of $\tau$ on $A$ if 
there exist indices $1 \leq i_1<i_2<\cdots<i_{|A|} \leq |B|$ such that
    \[
    (\tau(i_1),\tau(i_2),\ldots,\tau(i_{|A|}))=(\sigma(1),\sigma(2),\ldots,\sigma(|A|)).
    \]
\end{definition}
For example, if  $\tau=543621\in S_6$, 
then $\sigma=462\in S_{\{2,4,6\}}$ is a subpermutation of $\tau$ on $\{2,4,6\}$.

\section{$k$ Rotations Avoiding a Common Pattern}\label{Section:SinglePattern3}
By Theorem~\ref{Theorem:ClassicalSingle3}, for all $n\geq1$ and $q\in S_3$, we have $|S_n^{(1)}(q)|=|S_n(q)|=C_n$. In this section we enumerate $S_n^{(k)}(q)$ for $k\in\{2,3,\ldots,n\}$ and $q\in S_3$. We start with monotone patterns.

\begin{theorem}\label{Theorem:MonotonePatterns}
    For all $k$ with $2\leq k\leq n$, we have 
    \[
    |S_n^{(k)}(123)|=|S_n^{(k)}(321)|=2^{n-k+1}+k-2.
    \]
\end{theorem}
\begin{proof}
Let $n\geq2$. Since $(123)^r=321$, by Theorem~\ref{Theorem:ComplementReverse}, it suffices to prove the result for $S_n^{(k)}(123)$.

We first prove this for $k=2$. Let $\mathcal{A}\subseteq S_n$ be such that for all $p\in\mathcal{A}$ with $p(1)=\ell$, the subpermutation of $p$ on $\{\ell+1,\ell+2,\ldots,n\}$ is $(n,n-1,\ldots,\ell+1)$ and the subpermutation of $p$ on $\{\ell,\ell-1,\ldots,1\}$ is $(\ell,\ell-1,\ldots,1)$. 
According to this 
definition, for each fixed $\ell$, there are $\binom{n-1}{\ell-1}$ permutations $p\in\mathcal{A}$ with $p(1)=\ell$. Adding the permutations in $\mathcal{A}$ by their leading terms, we have 
\[
|\mathcal{A}|=\sum_{\ell=1}^n\binom{n-1}{\ell-1}=2^{n-1}.
\]

We will show that $S_n^{(2)}(123)=\mathcal{A}$. We first prove that $\mathcal{A}\subseteq S_n^{(2)}(123)$. Let $p\in\mathcal{A}$ with $p(1)=\ell$. 
Suppose, by way of contradiction, that $abc$ is a $123$ pattern in $p$. By the definition of $\mathcal{A}$, we cannot have $a>\ell$ or $c\leq\ell$ because otherwise $abc$ would be either a subpermutation of $(n,n-1,\ldots,\ell+1)$ or $(\ell,\ell-1,\ldots,1)$, 
a contradiction. It follows that $a\leq\ell$ and $c>\ell$. Now if $b\leq \ell$, then $ab$ is a subpermutation of $(\ell,\ell-1,\ldots,1)$, which is a contradiction; and if $b>\ell$, then $bc$ is a subpermutation of $(n,n-1,\ldots,\ell+1)$, which is again a contradiction. This proves that $p\in S_n^{(1)}(123)$. 
Notice this automatically implies that $p(2)p(3)\cdots p(n)$ avoids the pattern $123$. Next we show that $p^{(2)}$ avoids $123$. 
Suppose, by way of contradiction, that a subpermutation $abc$ of $p^{(2)}$ is a $123$ pattern. 
Since $p^{(2)}=p(2)p(3)\cdots p(n)\ell$ and $p(2)p(3)\cdots p(n)$ avoids $123$, we must have $c=\ell$. It follows that $a<b<\ell$ and hence $ab$ is a subpermutation of $(\ell-1,\ell-2,\ldots,1)$, which is a contradiction. This proves that $p \in S_n^{(2)}(123)$. Hence, we have $\mathcal{A}\subseteq S_n^{(2)}(123)$.

It remains to show that $S_n^{(2)}(123)\subseteq\mathcal{A}$. Let $p\in S_n^{(2)}(123)$. Suppose $p(1)=\ell$. Then the subpermutation of $p$ on $\{\ell+1,\ell+2,\ldots,n\}$ must be $(n,n-1,\ldots,\ell+1)$ because otherwise $p$ contains a $123$ pattern. At the same time, the subpermutation of $p$ on $\{\ell,\ell-1,\ldots,1\}$ is $(\ell,\ell-1,\ldots,1)$ because otherwise the subpermutation of $p^{(2)}$ on $\{\ell,\ell-1,\ldots,1\}$ would contain a 123 pattern. Hence, we have $S_n^{(2)}(123)\subseteq\mathcal{A}$. This completes the proof for $k=2$.

Now let $k\geq3$. Notice that $S_n^{(k)}(123)\subseteq S_n^{(2)}(123)=\mathcal{A}$. In particular, for all $p\in S_n^{(k)}(123)$ with $p(1)=\ell$, the subpermutation of $p$ on $\{\ell+1,\ell+2,\ldots,n\}$ is $(n,n-1,\ldots,\ell+1)$ and the subpermutation of $p$ on $\{\ell,\ell-1,\ldots,1\}$ is $(\ell,\ell-1,\ldots,1)$. Let $p\in S_n^{(k)}(123)$ with $p(1)=\ell$.

\textbf{Case 1}: $1\leq \ell\leq k-2$.
In this case, we show that $p=(\ell,\ell-1,\ldots,1,n,n-1,\ldots,\ell+1)$. 
Since $S_n^{(k)}(123)\subseteq\mathcal{A}$, this holds when $\ell=1$. So we assume that $\ell\geq2$. Suppose, by way of contradiction, that $p\neq (\ell,\ell-1,\ell-2,\ldots,1,n,n-1,\ldots,\ell+1)$. Since $S_n^{(k)}(123)\subseteq\mathcal{A}$, there exists $i\leq \ell-1$ such that $n$ comes before $\ell-i$. Let $j\leq \ell-1$ be the smallest such that $n$ comes before $\ell-j$. Then $(\ell-j,\ell,n)$ is a subpermutation of $p^{(j+2)}$ 
and $(\ell-j,\ell,n)$ is a $123$ pattern, a contradiction. So for each $\ell\in\{1,2,\ldots,k-2\}$, there is a unique $p\in S_n^{(k)}(123)$ with $p(1)=\ell$.

\textbf{Case 2}: $\ell=n$. In this case, we have a unique $p=(n,n-1,\ldots,1)$.

\textbf{Case 3}: $\ell\in\{k-1,k,\ldots,n-1\}$. 
In this case, we have $p(1)p(2)\cdots p(k-1)=(\ell,\ell-1,\ldots,\ell-k+2)$. 
This is because otherwise 
$p^{(j+2)}$ would have a $123$ pattern for some $j\leq k-2$, similar to Case 1.
Since the subpermutation of $p$ on $\{1,2,\ldots,\ell-k+1\}$ is $(\ell-k+1,\ell-k,\ldots,1)$ and the subpermutation of $p$ on $\{\ell+1,\ell+2,\ldots,n\}$ is $(n,n-1,\ldots,\ell+1)$, there are $\binom{n-\ell+\ell-k+1}{\ell-k+1}=\binom{n-k+1}{\ell-k+1}$ possibilities for $p$. 
We still need to check that for all of these possibilities, $p,p^{(2)},\ldots,p^{(k)}$ all avoid $123$. Let $p$ be such a permutation and let $\alpha$ be the subpermutation of $p$ on $\{1,2,\ldots,\ell-k+1\}\cup\{\ell+1,\ell+2,\ldots,n\}$. When $i=1$, 
$p^{(1)}$ is the concatenation of two blocks: the first block is $(\ell,\ell-1,\ldots,\ell-k+2)$ and the second block is $\alpha$. 
When $i=k$, 
$p^{(k)}$ is the concatenation of two blocks: the first block is $\alpha$ and the second block is $(\ell,\ell-1,\ldots,\ell-k+2)$. When $i\in\{2,3,\ldots,k-1\}$, we have that $p^{(i)}$ is a concatenation of three blocks: the first block is $(\ell-i+1,\ell-i,\ldots,\ell-k+2)$, the middle block is $\alpha$, and the third block is $(\ell,\ell-1,\ldots,\ell-i+2)$. 
Now we show that $p^{(i)}$ avoids $123$ for $i\in\{2,3,\ldots,k-1\}$. The cases for $i=1$ and $i=k$ are then similar. Suppose $abc$ is a $123$ pattern of $p^{(i)}$ for some $i\in\{2,3,\ldots,k-1\}$. Notice that the subpermutation of $p^{(i)}$ on $\{1,2,\ldots,\ell-i+1\}$ is $(\ell-i+1,\ell-i,\ldots,1)$ and the subpermutation of $p^{(i)}$ on $\{\ell-i+2,\ell-i+3,\ldots,n\}$ is $(n,n-1,\ldots,\ell-i+2)$. If $a\leq \ell-i+1$, then $bc$ is a subpermutation of $(n,n-1,\ldots,\ell-i+2)$, which is a contradiction; and if $a>\ell-i+1$, then $abc$ is a subpermutation of $(n,n-1,\ldots,\ell-i+2)$, which is again a contradiction.

Adding all the contributions together, we have
\[
|S_n^{(k)}(123)|=(k-2)+1+\sum_{\ell=k-1}^{n-1}\binom{n-k+1}{\ell-k+1}=2^{n-k+1}+k-2
\]
for $k\geq3$.
\end{proof}
Next, we enumerate $S_n^{k}(q)$, where $q\in S_3$ is not monotone.
\begin{theorem}\label{Theorem:NonMonotonePatterns}
    For all $k$ with $2\leq k\leq n$, we have
    \[
|S_n^{(k)}(213)|=|S_n^{(k)}(231)|=|S_n^{(k)}(312)|=|S_n^{(k)}(132)|=k-2+\sum_{i=0}^{n-k+1}C_i.
    \]
\end{theorem}
\begin{proof}
Since $(213)^c=231$, $(213)^r=312$, and $(312)^c=132$, by Theorem~\ref{Theorem:ComplementReverse}, it suffices to prove the result for $S_n^{(k)}(213)$.

We first prove this for $k=2$. Let $\mathcal{B}\subseteq S_n$ be the set of permutations such 
that for all $p\in\mathcal{B}$ with $p(1)=\ell$, 
we have $p=(\ell,p(2),p(3),\ldots,p(n-\ell+1),1,2,\ldots,\ell-1)$, where $p(2)p(3)\cdots p(n-\ell+1)$ 
avoids $213$.  By Theorem~\ref{Theorem:ClassicalSingle3}, for each $\ell$ there are $C_{n-\ell}$ 
possibilities for $p(2)p(3)\cdots p(n-\ell+1)$. Hence we have
\[
|\mathcal{B}|=\sum_{\ell=1}^nC_{n-\ell}=\sum_{\ell=0}^{n-1}C_\ell.
\]

We will show that $S_n^{(2)}(213)=\mathcal{B}$. We first prove that $\mathcal{B}\subseteq S_n^{(2)}(213)$. Let $p\in\mathcal{B}$ with $p(1)=\ell$. We need to show that $p$ and $p^{(2)}$ both avoid $213$. 
Suppose, by way of contradiction, that $abc$ is a $213$ pattern in $p$. Then we have $b<a<c$. If $a=\ell$, 
then since $b<a$ and $c>a$, $c$ must come before $b$, which is a contradiction. 
If $a>\ell$, then since $c>a>\ell$ and $b$ comes before $c$, $abc$ is a subpermutation of $p(2)p(3)\cdots p(n-\ell+1)$, which is a contradiction. If $a<\ell$, then $abc$ is a subpermutation of $(1,2,\ldots,\ell-1)$, which is again a contradiction. Hence, we have $p\in S_n(213)$. We still need to show that $p\in S_n^{(2)}(213)$. 
Suppose that $abc$ is a $213$ pattern in $p^{(2)}=(p(2),p(3),\ldots,p(n-\ell+1),1,2,\ldots,\ell-1,\ell)$. Then $b<a<c$. If $c\leq\ell$, then $abc$ is a 
subpermutation of $(1,2,\ldots\ell-1,\ell)$, which is a contradiction. If $c>\ell$, then $abc$ is a 
subpermutation of $p(2)p(3)\cdots p(n-\ell+1)$. Since $p(2)p(3)\cdots p(n-\ell+1)$ is $213$ avoiding, we again have a contradiction. This proves that $\mathcal{B}\subseteq S_n^{(2)}(213)$.

Now we prove that $S_n^{(2)}(213)\subseteq\mathcal{B}$. Let $p\in S_n^{(2)}(213)$ with $p(1)=\ell$. Since $p^{(2)}$ avoids $213$, the subpermutation of $p$ on $\{1,2,\ldots,\ell-1\}$ must be $(1,2,\ldots,\ell-1)$. Otherwise, there would exist $a,b\in\{1,2,\ldots,\ell-1\}$ such that $a<b$ and $ba\ell$ is a subpermutation of $p^{(2)}$. Now let $x>\ell$ and $y<\ell$. If $y$ comes before $x$, then the subpermutation $\ell yx$ of $p$ would be a $213$ pattern. Hence $x$ comes before $y$. It follows that $p=(\ell,p(2),p(3),\ldots,p(n-\ell+1),1,2,\ldots,\ell-1)$. Since $p$ avoids $213$, we 
have that $p(2)p(3)\cdots p(n-\ell+1)$ avoids $213$. Hence we have $S_n^{(2)}(213)\subseteq\mathcal{B}$. This completes the proof for $k=2$.

Now suppose $k\geq3$. Note that $S_n^{(k)}(213)\subseteq S_n^{(2)}(213)=\mathcal{B}$. So for all $p\in S_n^{(k)}(213)$ with $p(1)=\ell$, we have $p=(\ell,p(2),p(3),\ldots,p(n-\ell+1),1,2,\ldots,\ell-1)$ 
where $p(2)p(3)\cdots p(n-\ell+1)$ avoids $213$. Let $p\in S_n^{(k)}(213)$ with $p(1)=\ell$.

\textbf{Case 1}: $n-\ell\leq k-2$. Then $n-\ell+1\leq k-1$. In this case, we must have $p(2)p(3)\cdots p(n-\ell+1)=(\ell+1,\ell+2,\ldots,n)$. 
Suppose not. Then there exist $x>y>\ell$ such that $x$ comes before $y$. 
Then $y\ell x$ is a $213$ pattern and a subpermutation of $p^{(i)}$ for some $i\leq k$,  a contradiction.

Hence, for each $\ell$ with $n-\ell\leq k-2$, there is a unique $p\in S_n^{(k)}(213)$ with $p(1)=\ell$. Note that if $n-\ell\leq k-2$, then $\ell\geq n-k+2$. So there are $k-1$ such values for $\ell$.

\textbf{Case 2}: $n-\ell\geq k-1$. 
Then, similar to Case 1, we have $p(2)p(3)\cdots p(k-1)=(\ell+1,\ell+2,\ldots,\ell+k-2)$. 
By Theorem~\ref{Theorem:ClassicalSingle3} there are $C_{n-(\ell+k-2)}$ possibilities for $p(k)p(k+1)\cdots p(n-\ell+1)$. We still need to prove that all these possibilities belong to $S^{(k)}_n(213)$. Let $p$ be such a permutation. Then $p=(\ell,\ell+1,\ldots,\ell+k-2,\alpha,1,2,\ldots,\ell-1)$ where $\alpha$ is a $213$-avoiding permutation of $\{\ell+k-1,\ell+k,\ldots,n\}$. It follows that $p^{(k)}=(\alpha,1,2,\ldots,\ell+k-2)$ and $p^{(i)}=(\ell+i-1,\ell+i,\ldots,\ell+k-2,\alpha,1,2,\ldots,\ell+i-2)$ for $i\in\{1,2,\ldots,k-1\}$. 
Suppose, by way of contradiction, that $abc$ is a $213$ pattern in $p^{(k)}$. Since $\alpha$ avoids $213$, we have $c\leq\ell+k-2$. Since $b<a<c$, we have that $abc$ is a subpermutation of $(1,2,\ldots,\ell+k-2)$, which is a contradiction. Now let $i\in\{1,2,\ldots,k-1\}$. 
Suppose, by way of contradiction, that $abc$ is a $213$ pattern in $p^{(i)}$. Similar to $p^{(k)}$, if $c\leq \ell+i-2$ then we would have a contradiction. So we assume that $c>\ell+i-2$.
Then $abc$ is a subpermutation of $(\ell+i-1,\ell+i,\ldots,\ell+k-2,\alpha)$. Since $\alpha$ avoids $213$, we have $a\leq\ell+k-2$ and hence $b<a\leq \ell+k-2$. This implies that $ab$ is a subpermutation of $(\ell+i-1,\ell+i,\ldots,\ell+k-2)$, which is a contradiction. Hence, in this case there are exactly $C_{n-(\ell+k-2)}$ possibilities for each possible value $\ell$. 
Note that if $n-\ell\geq k-1$, then $\ell\leq n-k+1$.

Adding all the possibilities together, we find
\[
|S_n^{(k)}(213)|=k-1+\sum_{\ell=1}^{n-k+1}C_{n-(\ell+k-2)}=k-1+\sum_{\ell=1}^{n-k+1}C_\ell=k-2+\sum_{\ell=0}^{n-k+1}C_\ell.
\]
This completes the proof.
\end{proof}
By Theorem~\ref{Theorem:NonMonotonePatterns}, if $q\in S_3$ is not monotone, then $|S_n^{(2)}(q)|$ is the sum of first $n$ Catalan numbers starting with $C_0$. Setting $k=n$, Theorems~\ref{Theorem:MonotonePatterns} and \ref{Theorem:NonMonotonePatterns} together recover the fact that $|S_n^{(n)}(q)|=n$ for all $q\in S_3$. By Theorems~\ref{Theorem:MonotonePatterns} and \ref{Theorem:NonMonotonePatterns}, there are two Wilf-equivalence 
classes for $k\geq2$: one contains monotone patterns while the other one contains nonmonotone patterns. Incidentally, these classes coincide with the Wilf-equivalence classes of permutations containing a given pattern of length three $r$ times, with $r\geq 1$ (see, for example, \cite{CG2025}).

\section{Avoiding Rotational 3-Chains}\label{Section:Rot3Chain}
Next we enumerate $S_n(q:q^{(2)}:q^{(3)})$ for $q\in S_3$. Recall that $S_n(q:q^{(2)}:q^{(3)})$ is the set of $p\in S_n$ such that $p$ avoids $q$, $p^{(2)}$ avoids $q^{(2)}$, and $p^{(3)}$ avoids $q^{(3)}$. Throughout this section, we assume that $n\geq3$.

We first note that a rotation of a permutation contains a pattern if and only if the permutation contains a rotation of the pattern. This allows us to rephrase results on pattern avoidance for circular permutations and use them for permutations avoiding rotational 3-chains.

\begin{lemma}\label{Lemma:AvoidingAllRotations}
Let $n,m\in\mathbb{N}$, $p\in S_n$, and $q\in S_m$. Then a rotation of $p$ contains $q$ as a pattern if and only if $p$ contains a rotation of $q$ as a pattern.
\end{lemma}

By Lemma~\ref{Lemma:AvoidingAllRotations} and the fact that $
\text{Av}_n[123]=\{[(n,n-1,\ldots1)]\}$ and $
\text{Av}_n[321]=\{[(1,2,\ldots,n)]\}$ as discussed in Section~\ref{Section:Intro}, permutations avoiding all rotations of given pattern of length three have unique structures.

\begin{lemma}\label{Lemma:PermutationsAvRotations3}
Let $n\geq1$. If $q\in \{123,231,312\}\subseteq S_3$ then $p\in S_n$ avoids all rotations of $q$ if and only if $p$ is a rotation of $(n,n-1,\ldots,1)$. Similarly, if $q\in \{321,213,132\}\subseteq S_3$ then $p\in S_n$ avoids all rotations of $q$ if and only if $p$ is a rotation of $(1,2,\ldots, n)$.
\end{lemma}

We first show that if $q\in S_3$ and $q^{(2)}$ is not monotone, then $|S_n(q:q^{(2)}:q^{(3)})|=5n-10$.
\begin{theorem}\label{Theorem:3ChainLinear}
For all $n\geq3$ and $q\in\{123,321,213,231\}$, we have
\[
|S_n^{(3)}(q:q^{(2)}:q^{(3)})|=5n-10.
\]
\end{theorem}
\begin{proof}
We will prove that $|S_n^{(3)}(123:231:312)|=5n-10$. Since
\[
(321:213:132)=((123)^c:(231)^c:(312)^c),
\]
\[
(213:132:321)=((312)^r:(231)^r:(123)^r),
\]
and
\[
(231:312:123)=((213)^c:(132)^c:(321)^c),
\]
other results in the statement are consequences of Theorem~\ref{Theorem:ComplementReverse}. 

It is easy to see that $S_3^{(3)}(123:231:312)=S_3\backslash\{123\}$ and hence $|S_n^{(3)}(123:231:312)|=5n-10$ holds for $n=3$. Now, we assume that $n\geq4$.

Let $p\in S_n^{(3)}(123;231;312)$. Write $\sigma=p(3)p(4)\cdots p(n)$. Since $p=p(1)p(2)\sigma$, $p^{(2)}=p(2)\sigma p(1)$, 
and $p^{(3)}=\sigma p(1)p(2)$, we have that $\sigma$ avoids $123$, $231$, and $312$. Let $\pi$ denote the decreasing permutation on $\{1,2,\ldots,n\}\backslash\{p(1),p(2)\}$. By Lemma~\ref{Lemma:PermutationsAvRotations3}, $\sigma$ is a rotation of $\pi$. There are five 
disjoint cases depending on $p(1)$ and $p(2)$.

{\bf Case 1}:
$p(1)\neq n$ and $p(2)\neq n$. In this case, we must have $p(1)>p(2)$ because otherwise $p(1)p(2)n$ would be a $123$ pattern in $p$. We also have $p(n)>p(1)$ because otherwise $np(n)p(1)$ would be a $312$ pattern in $p^{(3)}$.
Finally, we must have $p(1)=p(2)+1$. Indeed, if $p(1)>p(2)+1$, then since $p(n)>p(1)$ and $\sigma$ is a rotation of $\pi$, we have that $(p(2),p(2)+1,p(n))$ is a $123$ pattern in $p$, which is a contradiction. Note also that if $p(2)>1$, then $p(3)<p(2)$ because otherwise $p(2)p(3)1$ is a $231$ pattern in $p^{(2)}$. To summarize, we must have 
\[
p=(a+1,a,a-1,1,\ldots,n,n-1,\ldots,a+1),
\]
where $a \in \{1,2, \dots, n-2\}$. Hence, $p$ must be a rotation of $(n,n-1,\ldots,1)$. Since all the rotations of $(n,n-1,\ldots,1)$ avoids the rotational chain $(123;231;312)$ by Lemma~\ref{Lemma:PermutationsAvRotations3}, there are $n-2$ such permutations in this case.

{\bf Case 2}: $p(1)=n$ and $p(2)=n-1$. Then $\pi=(n-2,n-3,\ldots,1)$. Notice that $\sigma$ is allowed to be any rotation of $\pi$. To see this, let $x,y\in\{1,2,\ldots,n-2\}$. Then none of $nxy$, $(n-1,x,y)$, or $(n,n-1,x)$ is a $123$ pattern in $p$. 
Similarly, none of $(n-1,x,y)$, $xyn$, or $(n-1,x,n)$ is a $231$ pattern in $p^{(2)}$ and none of $(x,n,n-1)$, $xyn$, or $(x,y,n-1)$ is a $312$ pattern in $p^{(3)}$. Since there are $n-2$ rotations of $\pi$, we have $n-2$ possibilities for $p$.

{\bf Case 3}: $p(1)=n$ and $p(2)\neq n-1$. Then $p(n)>p(2)$ because otherwise $(p(2),n-1,p(n))$ would be a $231$ pattern in $p^{(2)}$.
We also have $p(3)=n-1$ or $p(3)<p(2)$ because otherwise $(p(2),p(3),n-1)$ would be a $123$ pattern in $p$. To summarize, $p$ must be of the form \[\alpha_a:=(n,a,a-1,\ldots,1,n-1,n-2,\ldots,a+1),\]
with $a\in\{1,2,\ldots,n-2\}$. We still need to show that $\alpha_a\in S_n^{(3)}(123:231:312)$ for all $a\in\{1,2,\ldots,n-2\}$. To see this, let $a\in\{1,2,\ldots,n-2\}$ and $x,y\in\{1,2,\ldots,n-1\}\backslash\{a\}$. Since $\alpha_a$ is a cancatenation of two decreasing permutations, $\alpha_a$ does not contain the pattern $123$. Since $\alpha_a^{(2)}=\beta n$, where $\beta$ is a rotation of $(n-1,n-2,\ldots,1)$, we see that $\alpha_a^{(2)}$ does not contain the pattern $231$. Finally, let $xy$ be a 
subpermutation of $\alpha_a^{(3)}(1)\alpha_a^{(3)}(2)\cdots\alpha_a^{(3)}(n-2)$. If $x<a$ then none of $xya$, $xyn$, $xna$, or $yna$ is a $312$ pattern; and if $x>a$ then $y>a$ by the form of $\alpha_a$ and hence none of $xya$, $xyn$, $xna$, or $yna$ is a $312$ pattern. 
Therefore we have $n-2$ possibilities for $p$.

{\bf Case 4}: $p(1)=n-1$ and $p(2)=n$. Then $\pi=(n-2,n-3,\ldots,1)$. Similar to Case 2, $\sigma$ is allowed to be any rotation of $\pi$ and hence we have $n-2$ possibilities for $p$.

{\bf Case 5}: $p(1)\neq n-1$ and $p(2)=n$. Then $p(n)>p(1)$ because otherwise $(n-1,p(n),p(1))$ would be a $312$ pattern in $p^{(3)}$. We also have $p(3)=n-1$ or $p(3)<p(1)$ because otherwise $(p(1),p(3),n-1)$ would be a $123$ pattern in $p$. Hence, $p$ must be of the form
\[p=(a,n,a-1,a-2,\ldots,1,n-1,n-2,\ldots,a+1),
\]
where $a\in\{1,2,\ldots,n-2\}$. Using a similar argument as in Case 3, we can see that all these permutations avoid the rotational 3-chain $(123:231:312)$. Therefore, we have $n-2$ possibilities for $p$ in this case.

Adding all the possibilities together, we have \[|S_n^{(3)}(123;231;312)|=(n-2)+(n-2)+(n-2)+(n-2)+(n-2)=5n-10.\]

This completes the proof.
\end{proof}

We now enumerate $S_n(q:q^{(2)}:q^{(3)})$ for $q\in S_3$, where $q^{(2)}$ is monotone
and show that in contrast to Theorem~\ref{Theorem:3ChainLinear}, we now have $|S_n(q:q^{(2)}:q^{(3)})|=\half (n+7)(n-2)$.

\begin{theorem}\label{Theorem:3ChainQuadratic}
    For all $n\geq3$, we have
    \[
    |S_n^{(3)}(312:123:231)|=|S_n^{(3)}(132:321:213)|=\half (n+7)(n-2).
    \]
\end{theorem}
\begin{proof}
    We will prove that $|S_n^{(3)}(312:123:231)|=\half (n+7)(n-2)$. Since \[(132:321:213)=((312)^c:(123)^c:(231)^c),\] by 
Theorem~\ref{Theorem:ComplementReverse}, we will also have $|S_n^{(3)}(132:321:213)|=\half (n+7)(n-2)$. 
    
    It is easy to see that $S_3^{(3)}(312:123:231)=S_3\backslash\{312\}$ and hence $|S_n^{(3)}(312:123:231)|=\half (n+7)(n-2)$ holds for $n=3$. Now we assume that $n\geq4$.

    Let $p\in S_n^{(3)}(312;123;231)$. Write $\sigma=p(3)p(4)\cdots p(n)$. Since $p=p(1)p(2)\sigma$, $p^{(2)}=p(2)\sigma p(1)$, and $p^{(3)}=\sigma p(1)p(2)$, $\sigma$ avoids $312$, $123$, and $231$. Let $\pi$ denote the decreasing permutation on $\{1,2,\ldots,n\}\backslash\{p(1),p(2)\}$. By Lemma~\ref{Lemma:PermutationsAvRotations3}, $\sigma$ is a rotation of $\pi$.  There are six 
disjoint cases depending on $p(1)$, $p(2)$, and $p(3)$.

    {\bf Case 1}: $p(1)=n$. In this case we have $p(2)=n-1$, because otherwise $(p(1),p(2),n-1)$ would be a $312$ pattern in $p$. 
At the same time $\sigma=\pi$, because otherwise $\sigma$ has an increasing subpermutation $xy$ and then $nxy$ would be a $312$ pattern in $p$. Hence there is a unique permutation $p=(n,n-1,\ldots,1)$.

    {\bf Case 2}: $p(2)=n$. In this case, $\sigma=\pi$ because otherwise, similar to Case 1, $p$ would have a $312$ pattern in $p$. Hence, $p$ must be of the form $(n-1,n,n-2,n-3,\ldots,1)$ or \[\alpha_a:=(a,n,n-1,\ldots,a+1,a-1,a-2,\ldots,1),\] with $a\in\{1,2,\ldots,n-2\}$. Similar to Case 3 in the proof of Theorem~\ref{Theorem:3ChainLinear}, we can see that $\alpha_a\in S_n^{(3)}(312:123:231)$ for all $a\in\{1,2,\ldots,n-2\}$ by checking the cases. Hence, there are $n-1$ possibilities for $p$.

    {\bf Case 3}: $p(1)>p(2)$ and $p(3)=n$. In this case, we must have $p(1)=p(2)+1$ because otherwise $(p(1),p(2),p(2)+1)$ would be a $312$ pattern in $p$. Since $p(3)=n$ and $\sigma$ is a rotation of $\pi$, we must have $\sigma=\pi$. Hence, $p$ must be of the form
    \[
    \beta_a:=(a+1,a,n,n-1,\ldots,a+2,a-1,a-2,\ldots,1),
    \]
    with $a\in\{1,2,\ldots,n-2\}$. Again, as in Case 3 in the proof of Theorem~\ref{Theorem:3ChainLinear}, we see that $\beta_a\in S_n^{(3)}(312:123:231)$ for all $a\in\{1,2,\ldots,n-2\}$ by checking the cases. 
Hence, there are $n-2$ possibilities for $p$.

    {\bf Case 4}: $n>p(1)>p(2)$ and $p(3)\neq n$. Similar to Case 3, we must have $p(1)=p(2)+1$ because otherwise $(p(1),p(2),p(2)+1)$ would be a $312$ pattern in $p$. If $p(2)=1$, then $p=(2,1,n,n-1,\ldots,3)$ because otherwise $p^{(2)}$ would have a $123$ pattern. Now suppose $p(2)\neq 1$. Then we must have $p(3)<p(2)$ because otherwise $p(2)p(3)n$ would be a $123$ pattern in $p^{(2)}$. Also note that if $p(1)\neq n-1$, then we have $p(n)>p(1)$. This is because if $p(n)<p(1)$, then, since $\sigma$ is a rotation of $\pi$, we have $p(3)<p(n)$ and hence $p(3)p(n)p(1)$ is $123$ pattern in $p^{(2)}$, which is a contradiction. 
Hence
    \[
    p=(a+1,a,a-1,\ldots,1,n,n-1,\ldots,a+2),
    \]
    with $a\in\{1,2,\ldots,n-3\}$. Since $p$ is a rotation of $(n,n-1,\ldots,1)$, by Lemma~\ref{Lemma:PermutationsAvRotations3}, we have $p\in S_n^{(3)}(312:123:231)$ and hence there are $n-3$ choices for $p$ in this case.

    {\bf Case 5}: $p(1)<p(2)$ and $p(3)=n$. Since $\sigma$ is a rotation of $\pi$, we have $\sigma=\pi$. It follows that $p$ must be of the form
    \[
    \gamma_{ij}=(i,j,n,n-1,\ldots,j+1,j-1,j-2,\ldots,i+1,i-1,i-2,\ldots,1),
    \]
    where $i\in\{1,2,\ldots,n-2\}$ and $j\in\{i+1,\ldots,n-1\}$.
    The number of such permutations is
    \[
    \sum_{i=1}^{n-2}\sum_{j=i+1}^{n-1}1=\sum_{i=1}^{n-2}(n-1-i)=\half (n-1)(n-2).
    \]
    We still need to show that all these permutations belong to $S_n^{(3)}(312:123:231)$. Let $i\in\{1,2,\ldots,n-2\}$ and $j\in\{i+1,i+2,\ldots,n-1\}$. Notice that $\gamma_{ij}=ij\pi$. By Lemma~\ref{Lemma:PermutationsAvRotations3}, $\pi$ avoids the patterns $312$, $123$, and $231$. It remains to show that if $abc$ is a subpermutation of $ij\pi$ containing $i$ or $j$, then $abc$ is not a $312$ pattern; and similar for $\gamma_{ij}^{(2)}=j\pi i$ and $\gamma_{ij}^{(3)}=\pi ij$. Let $x,y\in\{1,2,\ldots,n\}\backslash\{i,j\}$ with $x>y$. Since $\pi$ is decreasing, $xy$ is a subpermutation of $\pi$. Since $i<j$, neither $ijx$ nor $ijy$ is a $312$ pattern in $\gamma_{ij}$. Since $x>y$, neither $ixy$ nor $jxy$ is a $312$ pattern in $\gamma_{ij}$. Since $i<j$ and $x>y$, in $\gamma_{ij}^{(2)}$, none of $jxi$, $jyi$, $xyi$ and $jxy$ is a $123$ pattern. Similarly, in $\gamma_{ij}^{(3)}$, none of $xyi$, $xyj$, $xij$, and $yij$ is a $231$ pattern. Hence $\gamma_{ij}\in S_n^{(3)}(312;123;231)$ for all $i\in\{1,2,\ldots,n-2\}$ and $j\in\{i+1,i+2,\ldots,n-1\}$. 
It follows that there are $\half (n-1)(n-2)$ possibilities for $p$.
    
    {\bf Case 6}: $p(1)<p(2)<n$ and $p(3)\neq n$. Since $\sigma$ is a rotation of $\pi$, we have $p(n)>p(3)$. 
It follows that $p(n)>p(2)$ because otherwise $p(2)p(3)p(n)$ would be a $312$ pattern in $p$. We also have $p(3)<p(1)$ because otherwise $p(3)p(n)p(1)$ would be a $231$ pattern in $p^{(3)}$. Since $p(n)>p(2)$, $p(3)<p(1)$, $p(1)<p(2)$, and $\sigma$ is a rotation of $\pi$, we must have $p(2)=p(1)+1$. So $p$ must be of the form
    \[
    \delta_a:=(a,a+1,a-1,a-2,\ldots,1,n,n-1,\ldots,a+2),
    \]
    with $a\in\{2,3,\ldots,n-2\}$. Now we show that $\delta_a\in S_n^{(3)}(312:123:231)$ for all $a\in\{2,3,\ldots,n-2\}$. Let $a\in\{2,3,\ldots,n-2\}$ and $x,y\in \{1,2,\ldots,n\}\backslash\{a,a+1\}$ with $x\neq y$. In $\delta_a$, we see that $(a,a+1,x)$ is not a $312$ pattern. If $axy$ is a $312$ pattern, then $x<y<a$ and $axy$ is a subpermutation of $(a,a-1,\ldots,1)$, which is a contradiction. Similarly, $(a+1,x,y)$ is not a $312$ pattern either. This proves that $\delta_a$ avoids $312$. In $\delta_a^{(2)}$, we have that $(a+1,x,a)$ is not a $123$ pattern. If $(a+1,x,y)$ is a $123$ pattern, then $a+1<x<y$ and $xy$ is a subpermutation of $(n,n-1,\ldots,a+2)$,  a contradiction. Similarly, $xya$ is not a $123$ pattern either. This proves that $\delta_a^{(2)}$ avoids $123$. Finally, in $\delta_a^{(3)}$, we see that $(x,a,a+1)$ is not a $231$ pattern. If $xya$ is a $231$ pattern, then $y>x>a$ and $xy$ is a subpermutation of $(n,n-1,\ldots,a+2)$, which is a contradiction. Similary, $(x,y,a+1)$ is not a $231$ pattern in $\delta_a^{(3)}$ either. This completes the proof that $\delta_a\in S_n^{(3)}(312:123:231)$. Hence there are $n-3$ possibilities for $p$ in this case.

    Putting all the possibilities together, we have
    \[
    |S_n^{(3)}(312:123:231)|=1+(n-1)+(n-2)+(n-3)+\half (n-1)(n-2)+(n-3)=\half (n+7)(n-2).
    \]
    This completes the proof.
\end{proof}

\section{Avoiding $2$-Chains}\label{Section:2Chain}
In this section we enumerate $S_n^{(2)}(q_1:q_2)$ where $q_1,q_2\in S_3$, $q_1\neq q_2$, and $n\geq2$. Note that the cases when $q_1=q_2$ 
are special cases of Theorems~\ref{Theorem:MonotonePatterns} and \ref{Theorem:NonMonotonePatterns}. By Theorem~\ref{Theorem:ComplementReverse} we have \[|S_n^{(2)}(\sigma_1;\sigma_2)|=|S_n^{(2)}(\sigma_1^c:\sigma_2^c)|=|S_n^{(2)}(\sigma_2^r:\sigma_1^r)|=|S_n^{(2)}((\sigma_2^r)^c:(\sigma_1^r)^c)|.\] 
Hence, there are ten cases for $(q_1:q_2)$ with $q_1\neq q_2$:
\begin{itemize}
    \item $|S_n^{(2)}(123:132)|=|S_n^{(2)}(321:312)|=|S_n^{(2)}(231:321)|=|S_n^{(2)}(213:123)|$;
    \item $|S_n^{(2)}(123:213)|=|S_n^{(2)}(321:231)|=|S_n^{(2)}(312:321)|=|S_n^{(2)}(132:123)|$;
    \item $|S_n^{(2)}(123:231)|=|S_n^{(2)}(321:213)|=|S_n^{(2)}(132:321)|=|S_n^{(2)}(312:123)|$;
    \item $|S_n^{(2)}(123:312)|=|S_n^{(2)}(321:132)|=|S_n^{(2)}(213:321)|=|S_n^{(2)}(231:123)|$;
    \item $|S_n^{(2)}(123:321)|=|S_n^{(2)}(321:123)|$;
    \item $|S_n^{(2)}(132:213)|=|S_n^{(2)}(312:231)|$;
    \item $|S_n^{(2)}(132:231)|=|S_n^{(2)}(312:213)|$;
    \item $|S_n^{(2)}(132:312)|=|S_n^{(2)}(312:132)|=|S_n^{(2)}(213:231)|=|S_n^{(2)}(231:213)|$;
    \item $|S_n^{(2)}(213:132)|=|S_n^{(2)}(231:312)|$;
    \item $|S_n^{(2)}(213:312)|=|S_n^{(2)}(231:321)|$.
\end{itemize}

We will enumerate $S_n^{(2)}(q_1:q_2)$ for the first $2$-chain in each case. First of all, by the Erd\H{o}s-Szekeres theorem \cite[p. 467]{ErdosSzekeres1935}, we have $|S_n^{(2)}(123:321)|=0$ for all $n\geq6$. To enumerate the other nine cases, 
we use a result of Simion and Schmidt \cite{SimionSchmidt1985} on the number of permutations avoiding two different patterns of length three. 
Let
$S_n(\sigma_1,\sigma_2)$ denote the set of permutations $p\in S_n$ avoiding both $\sigma_1$ and $\sigma_2$.
\begin{theorem}{\rm\cite[Section 3]{SimionSchmidt1985}}\label{Theorem:Pairs3}
For all $n\geq1$,
    \[
   \begin{split}
|S_n(123,132)|=&|S_n(321,312)|=|S_n(123,213)|=|S_n(321,231)|=|S_n(132,213)|=|S_n(312,231)|\\=&|S_n(132,231)|=|S_n(312,213)|=|S_n(132,312)|=|S_n(213,231)|=2^{n-1},
   \end{split} 
    \]
    and
    \[
    |S_n(123,312)|=|S_n(321,132)|=|S_n(123,231)|=|S_n(321,213)|=\binom{n}{2}+1.
    \]
\end{theorem}

We start with the five cases where $|S_n^{(2)}(q_1:q_2)|=|S_n(q_1,q_2)|$.

\begin{proposition}\label{Prop:123;132}
For all $n\geq2$, we have $|S_n^{(2)}(123:132)|=|S_n^{(2)}(132:312)|=2^{n-1}$.
\end{proposition}
\begin{proof}
We prove that $|S_n^{(2)}(123:132)|=2^{n-1}$. The proof of $|S_n^{(2)}(132:312)|=2^{n-1}$ is similar.

Let $p\in S_n^{(2)}(123:132)$. Write $p(1)=\ell$. Since $p$ avoids $123$, the subpermutation on $\{\ell+1,\ell+2,\ldots,n\}$ is decreasing. Since $p^{(2)}$ avoids $132$, if $x>\ell$ and $y<\ell$, then $x$ comes before $y$. It follows that $p=\ell\alpha\beta$ where $\alpha$ is the decreasing permutation on $\{\ell+1,\ell+2,\ldots,n\}$ and $\beta\in S_{\ell-1}(123,132)$ if $\ell>1$. By Theorem~\ref{Theorem:Pairs3}, the number of permutations of the form $\ell\alpha\beta$ is
    \[
    1+\sum_{\ell=2}^n2^{(\ell-1)-1}=2^{n-1}.
    \]

It remains to show that every permutation of the form $\ell\alpha\beta$ as described in the previous paragraph belongs to $S_n^{(2)}(123:132)$. We first show that $\ell\alpha\beta$ avoids $123$. Suppose, by way of contradiction, that $abc$ is a $123$ pattern in $\ell\alpha\beta$. If $a=\ell$, then $bc$ is a subpermtuation of $\alpha$, which is impossible because $\alpha$ is decreasing. If $a>\ell$, then $abc$ is a subpermutation of $\alpha$, which again is impossible. Finally, if $a<\ell$, then $abc$ is a subpermutation of $\beta$ which contradicts the fact that $\beta$ avoids $123$. This proves that $\ell\alpha\beta$ avoids $123$. We still need to show that $\alpha\beta\ell$ avoids $132$. Suppose, by way of contradiction, that $xyz$ is a $132$ pattern in $\alpha\beta\ell$. If $x>\ell$, then $xyz$ is a subpermutation of $\alpha$, which contradicts the fact that $\alpha$ is decreasing. If $x<\ell$, then $y>z=\ell$ because $\beta$ avoids $132$; but this would mean that $\beta$ contains a term greater than $\ell$, which is impossible. Hence $\ell\alpha\beta\in S_n^{(2)}(123:132)$.
\end{proof}

\begin{remark}
By Theorem~\ref{Theorem:Pairs3} and Proposition~\ref{Prop:123;132}, $|S_n^{(2)}(123:132)|=|S_n(123,132)|$. However, $S_n^{(2)}(123:132)\neq S_n(123,132)$ in general. For example, $35421\in S_5^{(2)}(123:132)$ but $35421\notin S_n(123,132)$.
\end{remark}

\begin{proposition}
For all $n\geq2$, $|S_n^{(2)}(132:231)|=|S_n^{(2)}(213:312)|=2^{n-1}$.
\end{proposition}
\begin{proof}
We first prove that $|S_n^{(2)}(132:231)|=2^{n-1}$. Since $S_2^{(2)}(132:231)=S_2$, the statement holds for $n=2$. So we assume that $n\geq3$.

    Let $p\in S_n^{(2)}(132:231)$. Write $p(1)=\ell$. We first notice that $\ell>n-2$. This is because if $\ell\leq n-2$, then either $(\ell,n,n-1)$ is a $132$ pattern in $p$ or $(n-1,n,\ell)$ is a $231$ pattern in $p^{(2)}$. Hence, $p$ must be of the form $\ell\alpha$ where $\ell\in\{n-1,n\}$ and $\alpha$ is a permutation on $\{1,2,\ldots,n\}\backslash\{\ell\}$ avoiding $132$ and $231$. By Theorem~\ref{Theorem:Pairs3}, there are $2^{(n-1)-1}+2^{(n-1)-1}=2^{n-1}$ such permutations.

    It remains to show that all permutations of the form $\ell\alpha$ as described in the previous paragraph belong to $S_n^{(2)}(132:231)$. 
Suppose, by way of contradiction, that $abc$ is a $132$ pattern in $\ell\alpha$. Since $\ell=n-1$ or $n$, we have $a\neq \ell$. It follows that $abc$ is a $132$ pattern in $\alpha$, which is a contradiction. Similarly, $\alpha\ell$ avoids $231$. Hence $\ell\alpha\in S_n^{(2)}(132:231)$. Therefore, $|S_n^{(2)}(132:231)|=2^{n-1}$.

    The proof of $|S_n^{(2)}(213:312)|=2^{n-1}$ is similar. The difference is that if $p\in S_n^{(2)}(213:312)$, then $p(1)\in\{1,n\}$.
\end{proof}

\begin{proposition}
For all $n\geq2$, $|S_n^{(2)}(123:312)|=\binom{n}{2}+1$.
\end{proposition}
\begin{proof}
    Let $p\in S_n^{(2)}(123:312)$. Write $p(1)=\ell$. 
    
    First suppose $p(1)=\ell<n$. Since $p$ avoids $123$, the subpermutation on $\{\ell+1,\ell+2,\ldots,n\}$ is decreasing and the subpermutation on $\{1,2,\ldots,\ell-1\}$ is decreasing. Since $p^{(2)}$ avoids $312$, for all $x<\ell$ and $y>\ell$, $x$ comes before $y$ in $p$. So $p$ is of the form $(\ell,\ell-1,\ldots,1,n,n-1,\ldots,\ell+1)$ with $\ell\in\{1,2,\ldots,n-1\}$. Now let $\ell=n$. Then $p$ is of the form $n\alpha$ where $\alpha\in S_{n-1}(123,312)$. By Theorem~\ref{Theorem:Pairs3}, the total number of permutations of the forms $(\ell,\ell-1,\ldots,1,n,n-1,\ldots,\ell+1)$ with $\ell\in\{1,2,\ldots,n-1\}$ and $n\alpha$ with $\alpha\in S_{n-1}(123,312)$ is
    \[
    (n-1)+\left[\binom{n-1}{2}+1\right]=\binom{n}{2}+1.
    \]

    It remains to show that both types of permutations described above belong to $S_n^{(2)}(123:312)$. Let $\beta_\ell=(\ell,\ell-1,\ldots,1,n,n-1,\ldots,\ell+1)$ with $\ell\in\{1,2,\ldots,n-1\}$. Notice that $\beta_\ell$ is a rotation of $(n,n-1,\ldots,1)$. By Lemma~\ref{Lemma:PermutationsAvRotations3}, $\beta_\ell$ avoids $123$ and $312$. 
Suppose, by way of contradiction, that $abc$ is a $312$ pattern of $\beta_\ell^{(2)}$. Then we have $c=\ell$. It follows that $a>\ell>b$ and $a$ comes before $b$ in $\beta_{\ell}$, which is a contradiction. Hence $\beta_\ell\in S_n^{(2)}(123:312)$. Now let $\alpha\in S_{n-1}(123,312)$. Suppose, by way of contradiction, that $xyz$ is a $123$ pattern in $n\alpha$. Then $x\neq n$ and hence $xyz$ is a $123$ in $\alpha$, which is a contradiction. Similarly, $\alpha n$ does not contain the pattern $312$ either. Hence, $n\alpha\in S_n^{(2)}(123:312)$.
\end{proof}

The remaining four cases all have the property that $|S_n^{(2)}(q_1:q_2)|\neq|S_n(q_1,q_2)|$. The last two cases are especially interesting because $q_2=q_1^{(2)}$ for these two cases and they provide alternative 
interpretations to two known integer sequences.
\begin{proposition}\label{Prop:123;213}
For all $n\geq2$, $|S_n^{(2)}(123:213)|=2n-2$.
\end{proposition}
\begin{proof}
We first notice that $S_2^{(2)}(123:213)=S_2$ and
\[
S_3^{(2)}(123:213)=\{231,312,213,132\}.
\]
Hence the statement is true for $n=2$ and $n=3$. Now we assume that $n\geq4$.

    Let $p\in S_n^{(2)}(123:213)$. Write $p(1)=\ell$. We first show that $\ell\leq 3$. Suppose, by way of contradiction, that $\ell\geq4$. Since $p^{(2)}$ avoids $213$ and $p^{(2)}(n)=\ell$, the subpermutation on $\{1,2,3\}$ must be $123$. Then $p$ contains the pattern $123$, which is a contradiction. There are three cases left depending on the value of $\ell$.

    \textbf{Case 1}: $\ell=1$. Since $p$ avoids $123$, the subpermutation on $\{2,3,\ldots,n\}$ is $(n,n-1,\ldots,2)$. Since $(1,n,n-1,\ldots,2)$ avoids $123$ and $(n,n-1,\ldots,1)$ avoids $213$, we have $p=(1,n,n-1,\ldots,2)$.

    \textbf{Case 2}: $\ell=2$. Since $p$ avoids $123$, the subpermutation on $\{3,4,\ldots,n\}$ is $(n,n-1,\ldots,3)$. Now there are $n-1$ possibilities for the location of $1$. For each $i\in\{2,3,\ldots,n\}$, let $\alpha_i\in S_n$ be the permutation such that $\alpha_i(1)=2$, the subpermutation of $\alpha_i$ on $\{3,4,\ldots,n\}$ is $(n,n-1,\ldots,3)$, and $\alpha_i(i)=1$. We will show that $\alpha_i\in S_n^{(2)}(123:213)$ for all $i\in\{2,3,\ldots,n\}$. 
Suppose, by way of contradiction, that $abc$ is a $123$ pattern in $\alpha_i$. Since the subpermutation of $\alpha_i$ on $\{3,4,\ldots,n\}$ is $(n,n-1,\ldots,3)$, we have $a,b\in\{1,2\}$ and hence $a=1$ and $b=2$, which contradicts the fact that $2$ is the leading term. Now suppose, by way of contradiction, $xyz$ is a $213$ pattern in $\alpha_i^{(2)}$. Then $z\geq3$ and hence $x>2$, which again contradicts the fact that the subpermutation on $\{3,4,\ldots,n\}$ is $(n,n-1,\ldots,3)$. Therefore, there are exactly $n-1$ possibilities for $p$ in this case.
    
    \textbf{Case 3:} $\ell=3$. Since $p^{(2)}$ avoid $213$, the subpermutation on $\{1,2\}$ is $(1,2)$. Since $p$ avoids $123$, the subpermutation on $\{4,5,\ldots,n\}$ is decreasing. For all $x>3$, we must have that $x$ comes before $2$ because otherwise $12x$ would be a $123$ pattern. Hence, $p$ must satisfy $p(n)=2$ and the subpermutation of $p$ on $\{2,4,5,\ldots,n\}$ is $(n,n-1,\ldots,4,2)$. There are $n-2$ such permutations depending on the location of $1$. Similar to Case 2, 
all of these permutations belong to $S_n^{(2)}(123:213)$. Hence, there are $n-2$ such $p$ in this case.

    Adding up all the contributions, we have 
    \[
    |S_n^{(2)}(123:213)|=1+(n-1)+(n-2)=2n-2.
    \]
    This completes the proof.
\end{proof}

\begin{proposition}
    For all $n\geq2$, $|S_n^{(2)}(132:213)|=2n-2$.
\end{proposition}
\begin{proof}
    Let $p\in S_n^{(2)}(132:213)$. Write $p(1)=\ell$. 
Since $p$ avoids $132$, the subpermutation on $\{\ell+1,\ell+2,\ldots,n\}$ is $(\ell+1,\ell+2,\ldots,n)$. 
Since $p^{(2)}$ avoids $213$, the subpermutation on $\{1,2,\ldots,\ell-1\}$ is $(1,2,\ldots,\ell-1)$. Since $p(2)p(3)\cdots p(n)$ avoids both $132$ and $213$, for all $x,y>\ell$ and $u,v<\ell$ with $x\neq y$ and $u\neq v$, neither $u$ nor $v$ can be located between $x$ and $y$, and vice versa. Otherwise, we would have either a $132$ pattern or a $213$ pattern in $p(2)p(3)\ldots p(n)$. Hence $p$ must be of the form $\ell\alpha\beta$ where $\alpha=(\ell+1,\ell+2,\ldots,n)$ and $\beta=(1,2,\ldots,\ell-1)$, or $\alpha=(1,2,\ldots,\ell-1)$ and $\beta=(\ell+1,\ell+2,\ldots,n)$. If $\ell=1$ or $\ell=n$, then there is one possibility; and if $\ell\neq 1,n$, then there are two possibilities. Adding up all the contributions, we have $|S_n^{(2)}(132:213)|\leq 1+1+2(n-2)=2n-2$.

    We still need to verify that all the permutations $\ell\alpha\beta$ as defined in the previous paragraph belong to $S_n^{(2)}(132:213)$. 
Suppose, by way of contradiction, that $abc$ is a $132$ pattern in $\ell\alpha\beta$. If $a\geq\ell$, then $b>c>\ell$ and $bc$ is a subpermutation of $(\ell+1,\ell+2,\ldots,n)$, which is impossible. If $a<\ell$, then either $c>\ell$ or $c<\ell$. If $c<\ell$, then $abc$ is a subpermutation of $(1,2,\ldots,\ell-1)$, which is impossible; and if $c>\ell$ then $b>c$ and $bc$ is a subpermutation of $(\ell+1,\ell+2,\ldots,n)$, which is also impossible. Hence, $\ell\alpha\beta$ avoids the pattern $132$. 
Using a similar argument, one can also see that $\alpha\beta\ell$ avoids the pattern $213$.
\end{proof}

\begin{proposition}\label{Prop:123;231}
    For all $n\geq2$, we have
    \[
    |S_n^{(2)}(123:231)|=n + 2 \binom{n}{3}.
    \]
\end{proposition}
\begin{proof}
Since $|S_2^{(2)}(123:231)|=2$, the statement is true for $n=2$. 

Now we assume that $n\geq3$. Let $p\in S_n^{(2)}(123:231)$. Write $p(1)=\ell$. Since $p$ avoids $123$, the subpermutation of $p$ on $\{\ell+1,\ell+2,\ldots,n\}$ must be $(n,n-1,\ldots,\ell+1)$. There are four 
disjoint cases depending on $\ell$.

    \textbf{Case 1}: $\ell=1$. Then the only possibility for $p$ is $(1,n,n-1,\ldots,2)$. Since $(1,n,n-1,\ldots,2)$ is a rotation of $(n,n-1,\ldots,1)$, by Lemma~\ref{Lemma:PermutationsAvRotations3}, we have that $(1,n,n-1,\ldots,2)$ avoids $123$ and $231$. Hence, there is one possiblity for $p$ in this case.

    {\bf Case 2}: $\ell=2$. Then $p$ must be of the form $\alpha_i$ where $i=\{2,3,\ldots,n\}$, $\alpha_i(i)=1$, and the subpermutation of $\alpha_i$ on $\{3,4,\ldots,n\}$ is $(n,n-1,\ldots,3)$. Similar to Case 2 in the proof of Proposition~\ref{Prop:123;213}, one can see that $\alpha_i\in S_n^{(2)}(123:231)$ for all $i\in\{2,3,\ldots,n\}$. Hence, there are $n-1$ possibilities for $p$ in this case.

    {\bf Case 3}: $\ell=n$. Then $p$ is of the form $n\beta$ with $\beta\in S_{n-1}(123,231)$. By Theorem~\ref{Theorem:Pairs3}, there are $\binom{n-1}{2}+1$ choices for $\beta$. Since $n$ cannot be the first term of a $123$ pattern in $n\beta$ or the last term of a $231$ pattern in $\beta n$, we see that $n\beta\in S_n^{(2)}(123:231)$ if and only if $\beta\in S_{n-1}(123,231)$. Hence, there are $\binom{n-1}{2}+1$ possibilities for $p$ in this case.

    {\bf Case 4}: $n\geq4$ and $3\leq \ell\leq n-1$. Let $p'$ be the subpermutation of $p$ on $\{1,2,\ldots,\ell-1\}$. Notice that $p'\in S_{\ell-1}(123,231)$. There are three disjoint subcases depending on $p'$.

    {\it Subcase 4.1}: $p'=(\ell-1,\ell-2,\ldots,1)$. In this subcase, $p$ must be of the form $\ell\alpha p'\beta$ where $\alpha\beta=(n,n-1,\ldots,\ell+1)$ because otherwise $p^{(2)}$ would contain a $231$ pattern. There are $n-\ell+1$ such permutations. We need to show that these permutations belong to $S_n^{(2)}(123:231)$. 
Suppose, by way of contradiction, that $abc$ is a $123$ pattern in $\ell\alpha p'\beta$. If $a\geq\ell$, then $bc$ is a subpermutation of $\alpha\beta$, which is impossible; and if $a<\ell$, then either $ab$ is a subpermutation of $p'$ or $bc$ is a subpermutation of $\beta$, which are both impossible. This proves that $\ell\alpha p'\beta$ avoids $123$. Using a similar argument, one can also see that $\alpha p'\beta\ell$ avoids $231$. Hence, for every $\ell\in\{3,4,\ldots,n-1\}$, there are $n-\ell+1$ possibilities for $p$ in this subcase.
    
    {\it Subcase 4.2}: $p'(r)=\ell-1$ with $r\geq2$. We will show that we must have $p'=(r-1,r-2,\ldots,1,\ell-1,\ell-2,\ldots,r)$ in this case. Since $p'$ avoids $231$, for all $x\in\{p'(1),p'(2),\ldots,p'(r-1)\}$ and $y\in \{p'(r+1),p'(r+2),\ldots,p'(\ell-1)\}$, we must have $x<y$. 
It follows that $\{p'(1),p'(2),\ldots,p'(r-1)\}=\{1,2,\ldots,r-1\}$. Since $p'$ avoids $123$, $p'(1)p'(2)\cdots p'(r-1)$ must be decreasing.  
Hence $p'(1)p'(2)\cdots p'(r-1)=(r-1,r-2,\ldots,1)$. Since $r\geq2$, we see that $1$ comes before $\ell-1$. Now if $p'(r+1)p'(r+2)\cdots p'(\ell-1)$ is not decreasing, then we would have a $123$ pattern in $p'$, which is a contradiction. Hence $p'=(r-1,r-2,\ldots,1,\ell-1,\ell-2,\ldots,r)$. Now the only possible locations for $\ell+1,\ell+2,\ldots,n$ are before $p'(1)$ or between $p'(r-1)$ and $p'(r)$. 
Otherwise $p(2)p(3)\cdots p(n)$ either has a $231$ pattern or a $123$ pattern. To summarize, $p$ must be of the form $\ell\alpha_1\beta_1\alpha_2\beta_2$ where $\alpha_1\alpha_2=(n,n-1,\ldots,\ell+1)$, $r\in\{2,3,\ldots,\ell-1\}$, $\beta_1=(r-1,r-2,\ldots,1)$, and $\beta_2=(\ell-1,\ell-2,\ldots,r)$. There are $n-\ell+1$ ways to write $(n,n-1,\ldots,\ell+1)$ as $\alpha_1\alpha_2$, including the cases where $\alpha_1$ and $\alpha_2$ are empty permutations; and there are $\ell-2$ possible values for $r$. It follows that there are $(\ell-2)(n-\ell+2)$ such permutations for each fixed $\ell\in\{3,4,\ldots,n-1\}$. We now show that all these permutations belong to $S_n^{(2)}(123:231)$. 
Suppose, by way of contradiction, that $abc$ is a $123$ pattern in $\ell\alpha_1\beta_1\alpha_2\beta_2$. If $a\geq\ell$, then $c>b>\ell$ and hence $bc$ is a subpermutation of $\alpha_1\alpha_2$, which is a contradiction; and if $a<\ell$, then either $bc$ is a subpermutation of $\alpha_2$ or $abc$ is a subpermutation of $p'$, which are both impossible. This proves that $\ell\alpha_1\beta_1\alpha_2\beta_2$ avoids $123$. One can also see that $\alpha_1\beta_1\alpha_2\beta_2\ell$ avoids $231$ in a similar way. Hence, for every $\ell\in\{3,4,\ldots,n-1\}$, there are $(\ell-2)(n-\ell+2)$ possibilities for $p$ in this subcase.

    {\it Subcase 4.3}: $p'(1)=\ell-1$ but $p'\neq (\ell-1,\ell-2,\ldots,1)$. In this case, there exist $x<y<\ell-1$ such that $x$ comes before $y$. Now let $z\in\{\ell+1,\ell+2,\ldots,n\}$. If $z$ comes after $y$, then $xyz$ is a $123$ pattern; if $z$ is located between $\ell-1$ and $y$, then $(\ell-1,z,y)$ is a $231$ pattern. Hence the only possible locations for $\ell+1,\ell+2,\ldots,n$ are before $p'(1)=\ell-1$. To summarize, $p$ must be of the form $(\ell,\alpha,\ell-1,\beta)$ where $\alpha=(n,n-1,\ldots,\ell+1)$ and $\beta\in S_{\ell-2}(123,231)$ such that $\beta$ is not decreasing. By Theorem~\ref{Theorem:Pairs3}, there are $\binom{\ell-2}{2}+1-1=\binom{\ell-2}{2}$ such permutations. We now show that all these permutations belong to $S_n^{(2)}(123:231)$. 
Suppose, by way of contradiction, that $abc$ is a $123$ pattern in $(\ell,\alpha,\ell-1,\beta)$. If $a\geq\ell$, then $c>b>\ell$ and hence $bc$ is a 
subpermutation of $\alpha$ which is a contradiction; if $a=\ell-1$, then $bc$ is a subpermutation of $\beta$ with $b,c>\ell-1$, which contradicts that $\beta\in S_{\ell-2}$; and if $a<\ell-1$, then $abc$ is a subpermutation of $\beta$, which contradicts that $\beta$ avoids $123$. This proves that $(\ell,\alpha,\ell-1,\beta)$ avoids the pattern $123$. We still need to show that $(\alpha,\ell-1,\beta,\ell)$ avoids $231$. 
Suppose, by way of contradiction, that $abc$ is a $231$ pattern in $(\alpha,\ell-1,\beta,\ell)$. If $c=\ell$, then $\ell<a<b$ and hence $ab$ is a subpermutation of $\alpha$, which is a contradiction; if $c>\ell$, then $\alpha$ has a $231$ pattern which is impossible; and if $c<\ell$, then either $ab$ is a subpermutation of $(\alpha,\ell-1)=(n,n-1,\ldots,\ell+1,\ell-1)$ or $abc$ is a subpermutation of $\beta$, which are both impossible. This proves that $(\ell,\alpha,\ell-1,\beta)$ avoids the pattern $231$. Hence, for every $\ell\in\{3,4,\ldots,n-1\}$, there are $\binom{\ell-2}{2}$ possibilities for $p$ in this subcase.

    Adding all the possibilities together, we have
    \[
    \begin{split}
    |S_n^{(2)}(123:231)|=&1+(n-1)+\binom{n-1}{2}+1+\sum_{\ell=3}^{n-1}\left[(n-\ell+1)+(\ell-2)(n-\ell+1)+\binom{\ell-2}{2}\right]\\=&n + 2 \binom{n}{3}.
    \end{split}
    \]
This completes the proof.
\end{proof}
\begin{remark}\label{Remark:123;231}
The formula in Proposition~\ref{Prop:123;231} counts the integer sequence A116731 in OEIS \cite{OEIS}.
\end{remark}

\begin{proposition}\label{Prop:213;132}
    For all $n\geq2$, we have
    \[
    |S_n^{(2)}(213:132)|=(n+2)\cdot2^{n-3}.
    \]
\end{proposition}
\begin{proof}
    Let $p\in S_n^{(2)}(213:132)$. Write $p(1)=\ell$. Since $p$ avoids $213$, if $x<\ell$ and $y>\ell$, then $y$ comes before $x$. It follows that $p$ is of the form $\ell\alpha\beta$ where $\alpha$ is a permutation on $\{\ell+1,\ell+2,\ldots,n\}$ and $\beta$ is a permutation on $\{1,2,\ldots,\ell-1\}$. Since $\alpha\beta$ avoids both $213$ and $132$, $\alpha$ and $\beta$ both avoid the patterns $213$ and $132$. By Theorem~\ref{Theorem:Pairs3}, there are exactly $2^{n-\ell-1}$ possibilities for $\alpha$ and $2^{\ell-1-1}$ possibilities for $\beta$. So the total number of possible $\ell\alpha\beta$ is
    \[
    2^{n-1-1}+2^{n-1-1}+\sum_{\ell=2}^{n-1}2^{n-\ell-1}2^{\ell-1-1}=(n+2)\cdot2^{n-3}.
    \]
    The first two terms on the left hand side of the above expressions account for the cases when $\ell=1$ and $\ell=n$.

    It remains to show that all the permutations $\ell\alpha\beta$ defined above belong to $S_n^{(2)}(213:132)$. Suppose, by way of contradiction, that $abc$ is a $213$ pattern in $\ell\alpha\beta$. If $a=\ell$, then $b<\ell$ comes before $c>\ell$ in $\ell\alpha\beta$, which is a contradiction; if $a>\ell$, then since $c>a>\ell$, $\alpha$ has a $213$ pattern which is a contradiction; and if $a<\ell$, then by the definition of $\ell\alpha\beta$, we have $b,c<\ell$ and $\beta$ has a $213$ pattern, which is again a contradiction. This proves that $\ell\alpha\beta$ avoids $213$. Using a similar argument, we can also see that $\alpha\beta\ell$ avoids $132$.
\end{proof}

\begin{remark}
The formula in Proposition~\ref{Prop:213;132} counts the integer sequence A045623 in OEIS \cite{OEIS}.
\end{remark}

\section{Concluding Remarks}\label{Section:Conclude}
We studied patterns of length three for pattern avoidance in permutation rotations. A natural next step is to study similar 
questions for patterns of length four. Here we note that, for certain cases of the problems we studied in Sections~\ref{Section:SinglePattern3} and \ref{Section:Rot3Chain}, the Wilf-equivalence classes for patterns of length four are also entirely determined by complements and reverses. To demonstrate this we combine some computational results with Theorem~\ref{Theorem:ComplementReverse}.
\begin{table}[H]
\renewcommand{\arraystretch}{1.5}
\centering
\begin{tabular}{|l||c|c|}
\hline
$q$ &$|S^{(2)}_n(q)|$; $n=4,5,6,7$&$|S^{(3)}_n(q)|$; $n=4,5,6,7$\\\hline\hline
1234, 4321&22, 91, 408, 1936&21, 80, 323, 1366\\\hline
1243, 2134, 3421, 4312&22, 91, 410, 1977&21, 79, 314, 1315\\\hline
1324, 4231&22, 91, 408, 1938&21, 80, 329, 1449\\\hline
1342, 2431, 3124, 4213&22, 91, 413, 2028&21, 79, 320, 1402\\\hline
1423, 2314, 3241, 4132&22, 91, 409, 1958&21, 80, 327, 1425\\\hline
1432, 2341, 3214, 4123&22, 92, 425, 2129&21, 81, 340, 1549\\\hline
2143, 3412&22, 92, 426, 2142&21, 81, 342, 1575\\\hline
2413, 3142&22, 92, 424, 2108&21, 82, 535, 1649\\\hline
\end{tabular}
\caption{$|S^{(2)}_n(q)|$ and $|S^{(3)}_n(q)|$ for $q\in S_4$ and $n\in\{4,5,6,7\}$}\label{Table1}
\end{table}

By Table~\ref{Table1} and Theorem~\ref{Theorem:ComplementReverse}, we have eight Wilf-equivalence classes for $S^{(2)}_n(q)$ with $q\in S_4$ and for $S^{(3)}_n(q)$ with $q\in S_4$,  determined entirely by complements and reverses. 
It would be interesting to see if the 
same result holds for $S^{(k)}_n(q)$ for $k\geq4$. 

\begin{question}
    Is it true that for every $k\geq4$, there are eight Wilf-equivalence classes for $S^{(k)}_n(q)$ with $q\in S_4$? 
\end{question}

\begin{table}[H]
\renewcommand{\arraystretch}{1.5}
\centering
\begin{tabular}{|l||c|}
\hline
$q$ &$|S^{(4)}_n(q:q^{(2)}:q^{(3)}:q^{(4)})|; n=4, 5, 6, 7$\\\hline\hline
1234, 2341, 4321, 3214&23, 94, 369, 1327\\\hline
1432, 2143, 3412, 4123&23, 95, 386, 1446\\\hline
1243, 4312, 1342, 4213&23, 94, 363, 1239\\\hline
2431, 3124, 3421, 2134&23, 94, 361, 1209\\\hline
1324, 3241, 4231, 2314&23, 94, 370, 1352\\\hline
1423, 3142, 2413, 4132&23, 94, 370, 1357\\\hline

\end{tabular}
\caption{$|S^{(4)}_n(q:q^{(2)}:q^{(3)}:q^{(4)})|$ for $q\in S_4$ and $n\in\{4,5,6,7\}$}\label{Table2}
\end{table}

Considering a different direction, it is easier to determine the number of Wilf-equivalence classes for permutations avoiding rotational 4-chains. By Table~\ref{Table2} and Theorem~\ref{Theorem:ComplementReverse}, there are six Wilf-equivalence classes for $S^{(4)}_n(q:q^{(2)}:q^{(3)}:q^{(4)})$ when $q\in S_4$. Again, the Wilf-equivalence classes are determined entirely by complements and reverses. 
However, the enumeration of $S^{(4)}_n(q:q^{(2)}:q^{(3)}:q^{(4)})$ for $q\in S_4$ is still open.

\begin{question}
Is it possible to find exact formulas for $|S^{(4)}_n(q:q^{(2)}:q^{(3)}:q^{(4)})|$ with $n\geq4$ and $q\in S_4$?
\end{question}

\section*{Acknowledgments}
\"{O}. E\u{g}ecio\u{g}lu would like to acknowledge his sabbatical time at Reykjavik University in
2019 during which he had a chance to learn about the combinatorics of
pattern avoidance. Some of the results in this article were previously included in the doctoral dissertation of the second author \cite{Gaiser2024b}. 
The authors used SageMath for computational experiments and some of the SageMath code was developed with the help of Google AI tools. 
M. Yin was supported in part by the Simons Travel Support for Mathematicians Grant 00007227.

\end{document}